%% file: FAD_Bickel.tex
\documentclass[12pt]{amsart}
\usepackage{latexsym}
\usepackage{amssymb, amsmath}
\usepackage[left=3cm,right=3cm,top=3cm, bottom=3cm ]{geometry}

 \newtheorem{thm}{Theorem}[section]
 \newtheorem{lem}[thm]{Lemma}
 \newtheorem{defn}[thm]{Definition}
 \theoremstyle{remark}
 \newtheorem{rem}[thm]{Remark}
 \newtheorem*{ex}{Example}
 \numberwithin{equation}{section}

\input{def_latex}

\title{Fundamental Agler Decompositions}
\author{Kelly Bickel \\
Washington University\\
St. Louis, Missouri 63130}
\date{\today}

\email{kbickel@math.wustl.edu}
\thanks{Partially supported by NSF Grant DMS 0966845 and partially supported by an AAUW Dissertation Fellowship.}
\keywords{Agler kernels, Schur functions, bidisk}
\subjclass{Primary 47A57; Secondary 46C07}

\begin{document}

\maketitle

\begin{abstract} We use shift-invariant subspaces of the Hardy space on the bidisk to provide an
elementary proof of the Agler Decomposition Theorem. We observe that these shift-invariant subspaces
are specific cases of Hilbert spaces that can be defined from Agler decompositions and analyze the
properties of such Hilbert spaces. We then restrict attention to rational inner functions and show
that the shift-invariant subspaces provide easy proofs of several known results about decompositions
of rational inner functions. We use our analysis to obtain a result about stable polynomials on the
polydisk. \end{abstract}

\bibliographystyle{plain}

\section{Introduction }

In 1916, Pick considered the following interpolation problem: given $n$ points $\lambda^1, \dots, \lambda^n \in \mathbb{D}$
and $n$ points $\mu^1, \dots, \mu^n \in \mathbb{D}$, when is there a holomorphic $\phi: \mathbb{D} \rightarrow \overline{\mathbb{D}}$ such
that $\phi(\lambda^i) = \mu^i$ for $i = 1, \dots, n?$ He proved that such a $\phi$ exists if and only if
there is a positive kernel $K: \{1, \dots, n\} \times \{1, \dots, n\} \rightarrow \mathbb{C}$ such that
 \begin{equation*} 1-\mu^i
\bar{\mu}^j = (1-\lambda^i \bar{\lambda}^j) K(i,j).\end{equation*} 
(In this paper, we call $K: \Omega \times \Omega
\rightarrow \mathbb{C}$ a positive kernel if, for all finite sets $\{\lambda^1, \dots, \lambda^m \} \subseteq \Omega$, the matrix
with entries $K( \lambda^i,\lambda^j)$ is positive semidefinite. A kernel is called holomorphic if it is holomorphic in
the first variable and conjugate-holomorphic in the second variable.) Pick's condition is necessary because
for any holomorphic $\phi: \mathbb{D} \rightarrow \overline{\mathbb{D}}$, the function $K: \mathbb{D} \times \mathbb{D} \rightarrow \mathbb{C}$
defined by
\begin{equation*}K(z, w) := \frac{ 1
-\phi(z)\overline{ \phi(w)}}{1-z \bar{w}}\end{equation*} 
is a positive holomorphic kernel. In the late
1980's, Agler generalized Pick's result to the bidisk  in \cite{ag1}. He showed that given points
$\lambda^1, \dots, \lambda^n \in \mathbb{D}^2$ and $\mu^1, \dots, \mu^n \in \mathbb{D}$, there is a holomorphic
$\phi:\mathbb{D}^2 \rightarrow \overline{ \mathbb{D}}$ with $\phi(\lambda^i) = \mu^i$ for $i = 1, \dots, n$ if and
only if there are positive kernels $K_1, K_2: \{1, \dots, n\} \times \{1, \dots, n\} \rightarrow \mathbb{C}$ such that
\begin{equation}
\label{eqn1.2} 1-\mu^i \bar{\mu}^j = (1-\lambda_1^i \bar{\lambda}_1^j) K_2(i,j)+ (1-\lambda_2^i
\bar{\lambda}_2^j) K_1(i,j),\end{equation}
where each $\lambda^i=(\lambda^i_1, \lambda^i_2).$ 
Unlike the one-variable case, it is not immediate that
(\ref{eqn1.2}) is a necessary condition. This is the context of the Agler Decomposition Theorem. In
\cite{ag90}, Agler showed that for holomorphic $\phi: \mathbb{D}^2 \rightarrow \overline{\mathbb{D}},$ there are
positive holomorphic kernels $K_1,K_2: \mathbb{D}^2 \times \mathbb{D}^2 \rightarrow \mathbb{C}$ with
 \begin{equation} \label{eqn1.3} 1 -
\phi(z) \overline{ \phi(w)} = (1-z_1 \bar{w}_1) K_2(z,w) + (1-z_2
\bar{w}_2) K_1(z,w), \end{equation} 
for all $z, w \in \mathbb{D}^2.$ This is called an \emph{Agler decomposition of $\phi,$} and the kernels $(K_1,K_2)$
are called \emph{Agler kernels of $\phi$}. To make future calculations easier, we have reversed
the typical ordering of the kernels in (\ref{eqn1.3}). Agler's proof  was a nonconstructive separation
argument relying on the fact that such $\phi$ satisfy And{\^o}'s inequality. It was pointed out in
\cite{colwer94} (and details also appear in \cite{agmc_dv}, using \cite{hol82}) that
(\ref{eqn1.3}) is actually equivalent to And{\^o}'s inequality. Recall that the set of holomorphic
functions $\phi: \mathbb{D}^d \rightarrow \overline{\mathbb{D}}$ is called the \emph{Schur class} on $\mathbb{D}^d.$ It
follows from results about generalizing And{\^o}'s inequality to $\mathbb{D}^d$ that, for $d \ge 3,$ the obvious generalization of
(\ref{eqn1.3}) holds on a strict subset of the Schur functions, called the \emph{Schur-Agler
class} on $\mathbb{D}^d.$

 In the interim, there has been much interest in both analyzing Agler decompositions
on the bidisk (see \cite{bsv05}, \cite{colwer99}, \cite{kn10c}, \cite{kn08bs}) and better
understanding the Schur-Agler class on the polydisk (see \cite{adr08}, \cite{babo02}, \cite{bb11},
\cite{kn10a}, \cite{kn10d}).

In this paper, we address the structure of such Agler decompositions on the bidisk using the basic
theory of reproducing kernel Hilbert spaces. In the appendix, we catalog the results about
kernels and reproducing kernel Hilbert spaces used in the paper. 
For now, recall that if $K$ is a positive kernel, there is a Hilbert
space $\mathcal{H}(K)$ with $K$ as its reproducing kernel.

For Agler kernels $(K_1,K_2)$ of a Schur function $\phi$, we analyze $\mathcal{H}(K_1)$ and $\mathcal{H}(K_2).$ We
also consider the positive holomorphic kernel 
\begin{equation} \label{eqn:ker}  K_{\phi}(z, w) := \frac{ 1 -\phi(z)\overline{
\phi(w)}}{(1-z_1 \bar{w}_1)(1-z_2 \bar{w}_2)},\end{equation} 
for $z, w \in \mathbb{D}^2.$ 
The Hilbert space with the reproducing kernel $K_{\phi}$ is denoted $\mathcal{H}_{\phi}.$
For $\phi$ inner, it is well-known that  $\mathcal{H}_{\phi}$ is equal isometrically to $H^2(\mathbb{D}^2)
 \ominus \phi H^2(\mathbb{D}^2),$ where $H^2(\mathbb{D}^2)$ denotes the Hardy space on the bidisk.

In Section 2, we consider inner $\phi$ and introduce fundamental shift-invariant subspaces of $\mathcal{H}_{\phi}$
and hence, of $H^2(\mathbb{D}^2).$ These subspaces are special cases of spaces that appear
naturally in the theory of scattering systems and scattering-minimal unitary colligations; such
subspaces are discussed extensively by Ball-Sadosky-Vinnikov in \cite{bsv05}. Specifically, for $r=1,2,$ we
let $Z_r$ denote the coordinate function $Z_r(z_1,z_2)=z_r.$ We then let \emph{$S^{max}_1 $} 
denote the largest subspace in $\mathcal{H}_{\phi}$ invariant under
multiplication by $Z_1$ and let \emph{$S^{min}_2$} $= \mathcal{H}_{\phi} \ominus S^{max}_1.$
We define \emph{$S^{max}_2$} and \emph{$S^{min}_1$} analogously.

We show that these subspaces yield an elementary proof of the Agler Decomposition Theorem, which is
constructive for inner functions. The result is implied by analyses in \cite{bsv05}, and related
arguments appear in a recent paper by Grinshpan-Kaliuzhnyi-Verbovetskyi-Vinnikov-Woerdeman in
\cite{gkvw08}, who prove a generalization of the Agler Decomposition Theorem. Their arguments use
the theory of scattering systems and shift-invariant subspaces of scattering subspaces. We present
this proof separately because it removes the need for scattering systems and provides concrete
decompositions that are used in later sections. We also develop a uniqueness criterion for inner
functions and show that non-extreme functions never have unique Agler decompositions.

In Section 3, we observe that the spaces $S^{max}_r$ and $S^{min}_r$ are special cases of more
general objects. Specifically, if $\phi$ is a Schur function with Agler kernels $(K_1,K_2)$, we
define the following Hilbert spaces:
\begin{equation*} S^{K}_r := \mathcal{H} \left ( \frac{K_r(z, w) }{1-z_r\bar{w}_r} \right), \end{equation*}
for $r=1,2.$ It is not hard to show that, for any inner $\phi$, the spaces $S^{max}_r$ and $S^{min}_r$ satisfy backward-shift invariant properties, and $S^{min}_r$ is, in some sense, a minimal $S^{K}_r$ space. In Propositions \ref{prop3.1} and
\ref{prop3.2}, we show that these properties extend to general $S^K_r$ spaces. In particular, we prove that for general $\phi$, the associated $S^K_r$ spaces also possess backward-shift invariant properties and
contain minimal sets. In Theorem \ref{thm3.4}, we characterize when a Schur function $\phi$ possesses
Agler kernels arising from an orthogonal decomposition of $\mathcal{H}_{\phi}.$

In Section 4, we use $S^{max}_r$ to examine Agler decompositions of rational inner functions. Let
$\phi$ be rational inner, and let the degree of $\phi$ in the variable $z_r$ be $k_r$ for
$r=1,2.$ We denote this by $\deg \phi = (k_1,k_2)$ and $\deg_{r} \phi =k_r.$ It is known
that for all Agler kernels $(K_1,K_2)$ of $\phi$, each $\mathcal{H}(K_r)$ is finite dimensional.
Specifically, \begin{align*} \text{dim}( \mathcal{H}(K_1)) \le k_2(k_1+1) \text{ and } \text{dim} ( \mathcal{H}(K_2) ) \le
k_1(k_2+1).\end{align*} The finiteness condition was proved by Cole and Wermer in \cite{colwer99}, and the
specific dimension bounds were found by Knese in \cite{kn10a}. We provide an alternate short proof using
$S^{max}_1$ and $S^{max}_2.$

We then consider rational inner functions $\phi$ continuous on $\overline{\mathbb{D}^2}.$ In Proposition
\ref{prop4.1}, we consider and slightly extend analyses from \cite{bsv05} about the Hilbert spaces associated to $\phi$. We use those results to show that such $\phi$ have unique Agler decompositions if and only if they are functions of one variable. This result was originally proven by Knese in \cite{kn08ub} using alternate methods. In Proposition \ref{thm4.2}, we show that this property does not extend to all rational inner functions and construct rational inner functions of arbitrarily high degree with
unique Agler decompositions.

In the concluding section, we provide an application of the analysis of $\mathcal{H}_{\phi}$ in Proposition
\ref{prop4.1}. Specifically, recall that a polynomial in $d$ variables is called \emph{stable} if it has no zeros on
$\overline{\mathbb{D}^d}.$ We first generalize Proposition \ref{prop4.1} to the polydisk in Proposition
\ref{prop5.1}. We then use it to generalize a result of Knese in \cite{kn08bs}
characterizing stable polynomials in two complex variables to polynomials in $d$ complex variables.

\section{The Agler Decomposition Theorem} 
As we deal exclusively with the bidisk, we denote $H^{\infty}(\mathbb{D}^2)$, $H^2(\mathbb{D}^2),$  $L^{\infty}(\mathbb{T}^2)$, and $L^2(\mathbb{T}^2)$
by $H^{\infty}$, $H^2$, $L^{\infty},$ and $L^2$ and denote the closed unit ball of $H^{\infty}$ by $H^{\infty}_1,$  which is
equivalent to the Schur class on $\mathbb{D}^2$. 

Given a vector subspace $U$ of a Hilbert space $\mathcal{H}$,
we let $\overline{U}$ denote the closure of $U$ in $\mathcal{H}.$ Then, $\overline{U}$ is a Hilbert space that inherits the inner product of $\mathcal{H}.$ We 
also let $P_{V}$ denote the projection operator onto a closed subspace $V$ of $\mathcal{H}$ and let $M_{\psi}$ be the operator
of multiplication by a function $\psi.$ 

For $r=1,2,$ let $z_r$ denote the independent variable and $Z_r$ denote the coordinate function defined by $Z_r(z_1,z_2)=z_r.$ Define the following closed subspaces of $L^2$: 
\begin{align*}
L^2_{*-} &:= \big \{ f \in L^2: \hat{f}(n_1,n_2) =0 \text{ for } {n_2} \ge 0 \big \}, \\ 
L^2_{+-} &:= \big \{ f \in L^2 : \hat{f}(n_1,n_2) =0 \text{ for } n_1
< 0 \text{ or } n_2 \ge 0 \big \}, \\ 
L^2_{--} &:= \big \{ f \in L^2:  \hat{f}(n_1,n_2) =0 \text{ for } n_1 \ge 0 \text{ or } n_2 \ge 0
\big \}. \end{align*} Define $L^2_{-*}$ and $L^2_{-+}$ analogously. We also view $H^2$ as a subspace 
of $L^2$ in the usual way. In particular, each function  $f \in H^2$ is associated with the $L^2$ function  whose Fourier coefficients equal the Taylor coefficients of $f$ (for details, see \cite{rud69}). 
This associated $L^2$ function is also denoted as $f$.  
Then $H^2$ can be viewed as the space of functions:
\begin{equation} \label{eqn:H2} 
 \big \{ f \in L^2: \hat{f}(n_1,n_2) =0 \text{ for } n_1 < 0 \text{ or } n_2 <0 \big \}. \end{equation}
For $n_1, n_2 \in \mathbb{N}$, we let $\hat{f}(n_1,n_2)$ denote both the Fourier coefficient of the $L^2$ function and the Taylor coefficient of the associated holomorphic $H^2$ function.
It is worth noting that this identification is equivalent to associating an $H^2$ function on $\mathbb{D}^2$ with its a.e.-defined boundary value function on $\mathbb{T}^2.$

\begin{defn} Let $\phi \in H^{\infty}_1$  be inner; specifically, assume the radial boundary values of $\phi$ exist a.e. and  satisfy
\begin{equation*} \lim_{r \nearrow 1} \big | \phi \big (r e^{i \theta_1}, r e^{i \theta_2} \big) \big | = \big | \phi \big ( e^{i \theta_1},  e^{i \theta_2} \big) \big | = 1 \text{ a.e.} \end{equation*} 
 Recall that $\mathcal{H}_{\phi}$ is the Hilbert space with reproducing kernel given by (\ref{eqn1.3}). Then, let $S^{max}_1 $ denote the largest
subspace in $\mathcal{H}_{\phi}$ invariant under $ M_{Z_1},$ i.e. invariant under multiplication by the coordinate
function $Z_1.$ A simple application of Zorn's Lemma shows
such a subspace must exist. It is immediate that $S^{max}_1$ is a closed subspace of $\mathcal{H}_{\phi},$ and hence, of $H^2.$ Let
$S^{min}_2 = \mathcal{H}_{\phi} \ominus S^{max}_1,$ and define $S^{max}_2$
and $S^{min}_1$ analogously.
 \end{defn}

\begin{rem} \label{rem2.1.b} For $\phi$ inner, $M_{\phi}$ is an isometry on $L^2$ and $H^2.$ It is then easy to
verify that 
\begin{align}
 \mathcal{H}_{\phi} &= H^2 \ominus \phi H^2 \nonumber \\ &= H^2 \cap
\phi [L^2 \ominus H^2 ]\nonumber \\ &= \big \{ \phi f \in H^2: f \in L^2_{*-} \oplus L^2_{-+}\big \}.\label{eqn2.1}
\end{align} \end{rem}

Other closed subspaces of $H^2,$ such as $S^{max}_r$ and $S^{min}_r,$ can also be identified with closed subspaces of $L^2$. 
Then, establishing $M_{Z_r}$-invariance of the subspace on $\mathbb{D}^2$ is equivalent to establishing $M_{Z_r}$-invariance of the subspace on $\mathbb{T}^2.$
In particular, each $S^{max}_r$ can be viewed as the maximal subspace of $(\ref{eqn:H2})$  invariant under $M_{Z_r}.$

The remark and lemma below detail special cases of Theorem 5.5 and Proposition 5.11 of Ball-Sadosky-Vinnikov
in \cite{bsv05}. We include simple proofs. Some of the arguments originate in \cite{bsv05}, while
others are our own.

\begin{rem} \label{rem2.1} Let $\phi$ be inner and assume we have an orthogonal decomposition 
\begin{equation*} \mathcal{H}_{\phi} =
S_1 \oplus S_2, 
\end{equation*}
with  $Z_r S_r \subseteq S_r$  for  $r=1,2.$ It is
almost immediate that $S^{min}_r \subseteq S_r.$ Specifically, for $f \in S^{min}_1,$ we can write
$f = f_1 + f_2 \text{ for } f_r \in S_r.$ By the maximality of $S^{max}_2$, we have $f_2 \in S^{max}_2$, which implies $f \perp f_2.$ By assumption, $f_1 \perp f_2$, so that 
\begin{equation*}\|f_2\|^2 = \langle f_2, f_1 +f_2 \rangle = \langle f_2, f
\rangle = 0.\end{equation*}
 Thus, $f=f_1 \in S_1.$ Similarly, $S^{min}_2 \subseteq S_2.$ \end{rem}

\begin{lem} \label{lem2.1} Let $\phi \in H^{\infty}_1$ be inner. Then \begin{align*} S^{max}_1 &= H^2 \cap \phi
L^2_{*-} \qquad S^{min}_1 = \overline{ P_{H^2} \phi L^2_{+-}} \\ S^{max}_2 &= H^2 \cap \phi
L^2_{-*} \qquad S^{min}_2 = \overline{ P_{H^2} \phi L^2_{-+}} , \end{align*} and $S^{max}_r$ and
$S^{min}_r$ are invariant under $M_{Z_r}$ for $r=1,2.$ \end{lem}

\begin{proof} We prove the results for $S^{max}_1$ and $S^{min}_2.$ By definition, 
\begin{equation*} S^{max}_1 = \big \{ f \in
\mathcal{H}_{\phi} : Z_1^kf \in \mathcal{H}_{\phi}, \ \forall \ k \in \mathbb{N} \big \}.\end{equation*}
 Let $S_1$ denote the set $H^2
\cap \phi L^2_{*-}.$ By the characterization of $\mathcal{H}_{\phi}$ in (\ref{eqn2.1}), $S_1 \subseteq
\mathcal{H}_{\phi}$. Since $Z_1 S_1 \subseteq S_1,$ we have $S_1 \subseteq S^{max}_1.$ 

Assume $g \in
S^{max}_1.$ Then, $g \in \mathcal{H}_{\phi},$ and (\ref{eqn2.1}) implies that $g=\phi f,$ for $f \in L^2_{*-} \oplus L^2_{-+}.$ Proceeding towards a contradiction, assume $g \not \in S_1$.
Then, there is some $(n_1,n_2) \in \mathbb{Z}^2$  such that $\hat{f}(n_1,n_2) \ne 0$ and $n_2 \ge 0$. It is immediate that
\begin{equation*} Z_1^{|n_1|}g \not \in \mathcal{H}_{\phi}, \end{equation*}
which gives the contradiction. Thus, $S^{max}_1 = H^2
\cap \phi L^2_{*-},$ and so, $S^{max}_1$ is precisely the space of $L^2$ functions orthogonal to the
closure of 
\begin{equation*} (L^2 \ominus H^2 )+ \phi (H^2 \oplus L^2_{-+})
\end{equation*}
 in $L^2.$ Since $S^{max}_1$ is closed,
we can calculate \begin{align*} S^{min}_2& := \mathcal{H}_{\phi} \ominus S^{max}_1 \\ &= P_{\mathcal{H}_{\phi}} \big [
(S^{max}_1 )^{\perp} \big] \\ &= \overline{ P_{\mathcal{H}_{\phi}} \big [ (L^2 \ominus H^2)+ \phi (H^2 \oplus
L^2_{-+}) \big] } \\ &= \overline{ P_{\mathcal{H}_{\phi}} \phi L^2_{-+} } \\ &= \overline{ P_{ H^2} \phi
L^2_{-+} }, \end{align*} where the last equality follows because $\phi L^2_{-+} \perp \phi H^2.$ Now, define the set  
\begin{equation*} L := \{ f \in L^2_{-+}: \hat{f}(n_1,n_2)=0 \text{ for all but finitely many $n_1$}\}. \end{equation*}  
Then, $L$ is dense in $ L^2_{-+}.$ Define $V = P_{H^2} \phi L,$ and let $f \in L$. Then, there is some $M \in \mathbb{N}$ such that 
we can write $f(z)=\sum_{m=1}^M f_m(z_2) z_1^{-m}$ a.e. on $\mathbb{T}^2,$ where each $f_m \in H^2(\mathbb{T})$ and satisfies
\begin{equation*} f_m(z_2) \sim \sum_{n=0}^{\infty} \hat{f}(m,n) z_2^n. \end{equation*}
Then, $P_{H^2} ( \phi f ) = \sum_{m=1}^M P_{H^2} ( \phi f_m Z_1^{-m}).$ By explicit calculation of Fourier coefficients, one can obtain
\begin{equation*} P_{H^2} \big(\phi f_m Z_1^{-m} \big)(z) \sim \sum_{j,k \ge 0} \widehat{\phi f_m}(j+m,k)z_1^j z_2^k. \end{equation*}
 Viewing  $P_{H^2} \big(\phi f_m Z_1^{-m} \big)$ as a holomorphic function on $\mathbb{D}^2$ and analyzing Taylor coefficients shows:
\begin{equation*} P_{H^2} \big(\phi f_m Z_1^{-m} \big)(z) = (X_{1}^m \phi f_m)(z) = (X^m_{1}\phi)(z) f_m(z_2), \end{equation*}
for $z \in \mathbb{D}^2,$ where $X_{1}$ denotes the backward shift operator on $H^2$  in the
$z_1$ coordinate defined by:
\begin{equation*} (X_{1}g)(z) =  \frac{ g(z) - g(0,z_2)}{z_1}, \end{equation*}
 for $g \in H^2$, and $X^{m}_{1}\phi $ denotes the function obtained by applying that backward shift operator $m$ times to $\phi$. By examining $P_{H^2} f$, it is immediate that:
\begin{equation} \label{eqn:vset}  V \subseteq \Big \{ \sum_{m=1}^{M}  \big (X^{m}_{1}\phi \big )(z) f_{m}(z_2): M \in \mathbb{N}, \ f_{m} \in H^2(\mathbb{D}) \Big\}. \end{equation}
By selecting specific $f \in L$ and doing analogous calculations,  containment in the other direction is easy to show. Thus, as a space of holomorphic functions, 
$V$ equals the set in $(\ref{eqn:vset})$. This characterization implies $V$ is invariant under $M_{Z_2}.$ As  
\begin{equation*} S^{min}_2 = \overline{P_{H^2} \phi L} = \overline{V}, \end{equation*}
$S^{min}_2$ must be invariant under $M_{Z_2}$ as well. The results for $S^{max}_2$ and
$S^{min}_1$ follow by symmetry. \end{proof}

We now provide an elementary proof of the Agler Decomposition Theorem. This result was first proven
by Agler in \cite[Theorem 2.6]{ag90}.

\begin{thm} \label{thm2.1} For $\phi \in H^{\infty}_1$, there are positive holomorphic kernels $K_1,K_2: \mathbb{D}^2
\times \mathbb{D}^2 \rightarrow \mathbb{C}$ satisfying
 \begin{align*} 1-\phi(z)\overline{\phi(w)} = (1 - z_1
\bar{w}_1)K_2(z,w) +(1-z_2 \bar{w}_2)K_1(z, w), \end{align*}
for all $z,
w \in \mathbb{D}^2.$  \end{thm}

\begin{proof} Let $\phi$ be an inner function, and let $S_1$ and $S_2$ denote the subspaces $S^{max}_1$ and
$S^{min}_2$. Since $S_1$ and $S_2$ are closed subspaces of $\mathcal{H}_{\phi}$, it
follows from Theorem \ref{thmA.4} that they are reproducing kernel Hilbert spaces that inherit the $\mathcal{H}_{\phi}$ inner product and have reproducing kernels given
by 
\begin{equation*}L_{S_r}(z, w) = P_{S_r} \left [ \frac{1 -\phi( \cdot) \overline{\phi(w)}}{(1- \cdot \ \bar{w}_1)(1- \cdot \ \bar{w}_2)}\right ] (z),\end{equation*} 
for $r=1,2.$ By Lemma
\ref{lem2.1}, each $S_r$ is invariant under $M_{Z_r}$. As each $S_r$ inherits the $\mathcal{H}_{\phi}$ norm, and $\mathcal{H}_{\phi}$ inherits the $H^2$ norm, we have $\|M_{Z_r} \|_{S_r} = 1.$ Theorem \ref{thmA.5} implies
 \begin{align*} K_r(z, w)
&: = (1-z_r \bar{w}_r)L_{S_r}(z, w)\end{align*} 
is a positive kernel for $r=1,2.$ As the $S_r$ are
Hilbert spaces of holomorphic functions, it follows that the $K_r$ are holomorphic kernels. Since
$\mathcal{H}_{\phi} =S_1 \oplus S_2,$ we have
 \begin{align} \frac{1 -\phi(z)
\overline{\phi(w)}}{(1-z_1\bar{w}_1)(1-z_2\bar{w}_2)} 
&= L_{S_1}(z, w) +L_{S_2}(z, w) \nonumber \\ 
&= \frac{K_1(z, w) }{1-z_1\bar{w}_1} + \frac{K_2(z,w) }{1-z_2\bar{w}_2}. \label{eqnker}
\end{align} 
Rearranging terms shows that $(K_1,K_2)$ are Agler kernels of
$\phi$. Basic manipulations of (\ref{eqnker}) and an application of Corollary \ref{corA.1} show that each $\mathcal{H}(K_r)$ is 
contained contractively in $\mathcal{H}_{\phi}$ and hence, in $H^2.$

For $\phi$ not inner, Theorem 5.5.1 in \cite{rud69} gives a sequence of inner
functions $\{ \phi^n \}$ converging locally, uniformly to $\phi.$ Let $\{ K_1^n \}$ and $\{ K_2^n
\}$ denote sequences of Agler kernels of the $\{\phi^n\}$. For $r=1, 2$, the Cauchy-Schwarz inequality
and the contractive containment of the $\mathcal{H}(K^n_r)$ inside $H^2$ can
be used to show that
\begin{equation*} | K_r^n(z, w)|^2 \le \frac{1}{(1-|z_1|^2)(1-|z_2|^2)}
\frac{1}{(1-|w_1|^2)(1-|w_2|^2)},\end{equation*} 
for all $ z, w \in \mathbb{D}^2$ and $n \in \mathbb{N}.$ Since the sequences $\{K_r^n \}$ are locally, uniformly
bounded, they form a normal family. Then, there is a subsequence $\{ \phi^{n_k} \}$ such that the
associated kernel subsequences $\{ K_1^{n_k} \}$ and $\{ K_2^{n_k} \}$ converge locally uniformly to
positive holomorphic kernels $K_1$ and $K_2$ satisfying 
\begin{align*} 
1-\phi(z)\overline{\phi(w)} = (1 - z_1 \bar{w}_1)K_2(z,w) +(1-z_2
\bar{w}_2)K_1(z, w),  \end{align*} 
for all $z, w \in \mathbb{D}^2.$
\end{proof}

And\^o's inequality follows as a corollary of Theorem \ref{thm2.1}. For the finite case, the arguments appear in
\cite[Theorem 1.2]{colwer94}. The general case and a refined inequality are discussed in \cite{agmc_dv}, using \cite{hol82}.

\begin{cor}\emph{ (And\^o's Inequality)} Let $p$ be a polynomial in $H^{\infty}_1,$ and let
$(T_1,T_2)$ be a pair of commuting contractions on a Hilbert space $\mathcal{H}$. Then, $p(T_1,T_2)$ is 
a contraction on $\mathcal{H}$, i.e. 
\begin{equation*} \| p(T_1,T_2) \|_{\mathcal{H}} \le 1.\end{equation*} \end{cor}

Theorem \ref{thm2.1} provides simple Agler kernels for inner functions. For ease of
notation, we will often denote positive kernels $K(z,w)$ defined on $\mathbb{D}^2 \times \mathbb{D}^2$ by simply $K$.

\begin{rem} \label{rem2.2} Let $\phi \in H^\infty_1$ be inner. From the arguments in the proof of Theorem
\ref{thm2.1}, it is clear that there are positive holomorphic kernels, now denoted $K^{max}_r$ and $K^{min}_r$, such that
\begin{equation} \label{eqn:funker} S^{max}_r = \mathcal{H} \left ( \frac{ K^{max}_r}{1-z_r \bar{w}_r} \right ) \ \text{ and } \ S^{min}_r
= \mathcal{H} \left ( \frac{ K^{min}_r}{1-z_r \bar{w}_r} \right ),\end{equation} for $r=1,2.$ Moreover,
$(K^{max}_1, K^{min}_2)$ and $(K^{min}_1, K^{max}_2)$ are pairs of Agler kernels of $\phi.$ \end{rem}

Our proof of Theorem \ref{thm2.1} provides insight into the uniqueness of Agler decompositions for
inner functions. The following result generalizes part of Theorem 5.10 in \cite{bsv05}.

\begin{thm} \label{thm2.2} Let $\phi \in H^{\infty}_1$ be inner. Then $\phi$ has a unique Agler
decomposition if and only if 
\begin{equation*}\phi L^2_{--} \cap H^2 = \{0\}.\end{equation*} \end{thm}

\begin{proof} It follows from the definitions of $S^{max}_r$ and $S^{min}_r$ and from Lemma \ref{lem2.1} that 
\begin{align}
\label{eqn2.2} S^{max}_1 \ominus S^{min}_1 = S^{max}_2 \ominus S^{min}_2 =\phi L^2_{--} \cap H^2.
\end{align} 
($\Rightarrow$) Assume $\phi$ has a unique Agler decomposition. By Remark \ref{rem2.2},
this implies 
\begin{equation*} (K^{max}_1,K^{min}_2) = (K^{min}_1,K^{max}_2). \end{equation*} 
By the representation of
$S^{max}_r$ and $S^{min}_r$ in Remark \ref{rem2.2}, we must have $S^{max}_r = S^{min}_r.$ Using
(\ref{eqn2.2}), this implies $\phi
L^2_{--} \cap H^2 =\{0\}.$ \bigskip

\noindent
($\Leftarrow$) Assume $\phi L^2_{--} \cap H^2 = \{0\}$. Then by (\ref{eqn2.2}), each $S^{max}_r =S^{min}_r.$ 
As the kernels in $(\ref{eqn:funker})$ are obtained by projecting the kernel $K_{\phi}$ from (\ref{eqn:ker}) onto the associated
$S^{max}_r / S^{min}_r$  space, it follows that each $K^{max}_r = K^{min}_r.$ 
 In particular,
$(K^{min}_1, K^{min}_2)$ is a pair of Agler kernels of $\phi.$ Let $(L_1,L_2)$ be another pair of
Agler kernels. Then, we have 
\begin{align}
 \nonumber \frac{1
-\phi(z)\overline{\phi(w)}}{(1-z_1\bar{w}_1)(1-z_2\bar{w}_2)} &=
\frac{L_1}{1-z_1\bar{w}_1} + \frac{L_2}{1-z_2\bar{w}_2} \\ &=
\frac{K^{min}_1}{1-z_1\bar{w}_1} + \frac{K^{min}_2}{1-z_2\bar{w}_2}. \label{eqn2.4}
\end{align} 
For each fixed $w \in \mathbb{D}^2$ and $r=1,2$,
 \begin{equation*} \frac{K^{min}_r}{1-Z_r\bar{w}_r},
\frac{L_r}{1-Z_r\bar{w}_r} \in S^{max}_r = S^{min}_r.\end{equation*} As $S^{min}_1 \perp S^{min}_2,$ the
decomposition in (\ref{eqn2.4}) is unique for each $w.$ It follows that for $r=1,2,$\begin{equation*}
\frac{L_r}{1-Z_r\bar{w}_r}  =  \frac{K^{min}_r}{1-Z_r\bar{w}_r}.\end{equation*} As
$L_1=K^{min}_1$ and $L_2=K^{min}_2$, then $\phi$ has a unique Agler decomposition. \end{proof}

We also observe that certain functions have extremely non-unique Agler decompositions. Recall that a
function $\phi$ is an \emph{extreme point} of $H^{\infty}_1$ if and only if there is \textit{no}
$f \in H_1^{\infty}$ such that $\phi \pm f \in H^{\infty}_1.$

\begin{thm} \label{thm1.3} If $\phi \in H^{\infty}_1$ is not an extreme point, then $\phi$ does not have a
unique Agler decomposition. \end{thm} 

\begin{proof} Assume $\phi$ is not extreme. Then, there is some $f \in
H^{\infty}_1$ such that $\phi \pm f \in H^{\infty}_1$, and so there are Agler kernels $(K_1,K_2)$ and
$(L_1,L_2)$ satisfying
\begin{align} 1- (\phi+f)(z)\overline{(\phi+f)(w)} &=
(1-z_1\bar{w}_1)K_2 + (1-z_2\bar{w}_2)K_1, \label{eqn2.5}\\ 1-
(\phi-f)(z)\overline{(\phi-f)(w)} &= (1-z_1\bar{w}_1)L_2 +
(1-z_2\bar{w}_2)L_1. \label{eqn2.6} \end{align}
Adding (\ref{eqn2.5}) and (\ref{eqn2.6}) and dividing the resultant equation by $2$ yields \begin{align*} 1-
\phi(z)\overline{\phi(w)} - f(z)\overline{f(w)} =
(1-z_1\bar{w}_1)\tfrac{K_2+L_2}{2} + (1-z_2\bar{w}_2)\tfrac{K_1+L_1}{2},
\end{align*} which implies 
\begin{align*} 1- \phi(z)\overline{\phi(w)} = &
(1-z_1\bar{w}_1) \left( \frac{K_2+L_2}{2} + t\frac{
f(z)\overline{f(w)}}{1-z_1\bar{w}_1} \right )
\\ & +  (1-z_2\bar{w}_2) \left(
\frac{K_1+L_1}{2} 
 + (1-t)\frac{ f(z)\overline{f(w)}}{1-z_2\bar{w}_2} \right),\end{align*} for
any $t \in [0,1]$. Hence, $\phi$ has infinitely many Agler decompositions.\end{proof}

\section{Analysis of Agler Spaces}
In the previous section, we showed that for $\phi$ inner, the subspaces $S^{max}_r$ and $S^{min}_r$
yield simple Agler decompositions. In this section, we will use these subspaces to analyze the
properties of similar spaces associated to general Schur functions.

As before, for ease of notation, we will often denote kernels defined 
on the bidisk simply by $K,$ instead of by $K(z, w).$

\begin{defn} \label{rem3.1} Let $\phi \in H^{\infty}_1,$ and let $ (K_1, K_2)$ denote a
pair of Agler kernels of $\phi$. Define the Hilbert spaces
 \begin{equation*}S^{K}_1 := \mathcal{H} \left(\frac{K_1}{1-z_1 \bar{w}_1}
\right) \  \text{ and } \  S^{K}_2 :=\mathcal{H} \left (\frac{K_2}{1-z_2 \bar{w}_2} \right).\end{equation*} 
We call
$S^K_1$ and $S^K_2$ \emph{Agler spaces of $\phi$.} By definition, $(K_1,K_2)$ satisfy 
\begin{align} \label{eqn3.1}
1-\phi(z)\overline{\phi(w)} = (1 - z_1 \bar{w}_1)K_2 +(1-z_2 \bar{w}_2)K_1,
\end{align} 
which immediately implies 
\begin{equation*}\frac{1-\phi(z)\overline{\phi(w)}}{(1 - z_1
\bar{w}_1)(1-z_2 \bar{w}_2)} = \frac{K_1}{1-z_1 \bar{w}_1} +
\frac{K_2}{1-z_2 \bar{w}_2}.\end{equation*} 
Arithmetic and an application of Corollary \ref{corA.1} can
be used to show that $S^K_1$, $S^K_2$, $\mathcal{H}(K_1),$ and $\mathcal{H}(K_2)$ are all contractively contained in
$\mathcal{H}_{\phi}$ and hence, in $H^2$. Moreover, it follows from Theorem \ref{thmA.5} that $S_r$ is invariant
under $M_{Z_r},$ and $\|M_{Z_r}\|_{S_r} \le 1$ for $r=1,2.$ \end{defn}

By Remark \ref{rem2.2}, $S^{max}_r$ and $S^{min}_r$ are special cases of the $S^K_r$ spaces. We will
show that certain properties of $S^{max}_r$ and $S^{min}_r$ extend to general Agler spaces.  First, recall that 
$X_{r}$ denotes the backward shift operator in the $z_r$ coordinate for $r=1,2.$ Specifically
$X_1$ and $X_2$  are defined by
\begin{equation*} (X_1 g) (z) = \frac{g(z)-g(0,z_2)}{z_1}, \ \text{ and } \ (X_2 g) (z) = \frac{g(z)-g(z_1,0)}{z_2},  \end{equation*}
for $g \in H^2.$
We will use the following result of
Alpay-Bolotnikov-Dijksma-Sadosky from \cite[Theorem 2.5]{abds01}:

\begin{thm} \label{lem3.1.a} Let $\phi \in H_1^{\infty}.$ Then $\mathcal{H}_{\phi}$ is invariant under $X_{r}$
for $r=1,2,$ and for  $f \in \mathcal{H}_{\phi}, $
\begin{align*} \|X_{1} f \|^2_{\mathcal{H}_{\phi}} &\le \| f \|^2_{\mathcal{H}_{\phi}} - \|f(0,z_2)
\|^2_{H^2}, \\ \|X_{2} f \|^2_{\mathcal{H}_{\phi}} &\le \| f \|^2_{\mathcal{H}_{\phi}} - \|f(z_1,0)
\|^2_{H^2}. \end{align*} \end{thm}

Observe the following fact:

\begin{lem} \label{lem3.1} Let $\phi \in H^{\infty}_1$ be inner. Then, $S^{max}_1$ and $S^{min}_1$ are invariant under
$X_{2},$ and $S^{max}_2$ and $S^{min}_2$ are invariant under $X_{1}$. \end{lem}

\begin{proof} It follows from the arguments in Lemma \ref{lem2.1} that
\begin{align}
\label{eqn3.2a} S^{max}_1 &= \  \big \{ f \in \mathcal{H}_{\phi} : Z_1^kf \in \mathcal{H}_{\phi}, \ \forall \ k \in
\mathbb{N} \big \}, \\ 
\label{eqn3.2b} S^{min}_1 &=  \ clos_{H^2} \Big \{ \sum_{m=1}^M  \big (X^{m}_{2}\phi \big )(z) f_m(z_1): M \in \mathbb{N}, \ f_m \in H^2(\mathbb{D}) \Big\},\end{align}
where $clos_{H^2}$ indicates that we are taking the closure of the set in $H^2.$
 It follows from (\ref{eqn3.2a}) and the
$X_{2}$-invariance of $\mathcal{H}_{\phi}$ that $S^{max}_1$ is invariant under $X_{2}$. It is
clear from (\ref{eqn3.2b}) and the fact that $X_2$ is a contraction on $H^2$ that $S^{min}_1$ is invariant under $X_{2}$. The result follows
for $S^{max}_2$ and $S^{min}_2$ by symmetry. \end{proof}

We will show that the properties listed in Lemma \ref{lem3.1} also hold for general Agler spaces.  First, for $r=1,2,$ let $H^2_{r}$
denote the space $H^2(\mathbb{D})$ with independent variable $z_r.$ Specifically, we have
\begin{equation*} H^2_r = \mathcal{H} \left ( \frac{1}{1-z_r \bar{w}_r } \right). \end{equation*}

\begin{prop} \label{prop3.1} Let $\phi \in H^{\infty}_1$ and let $(K_1, K_2)$ be Agler kernels of
$\phi$. Then $S^{K}_1$ is invariant under $X_{2},$ and $S^{K}_2$ is invariant under
$X_{1}.$ Moreover, for all $ f \in S^{K}_2$ and $
g \in S^{K}_1$, 
\begin{align*} \|X_{1} f \|^2_{S^K_2} &\le \| f \|^2_{S^K_2} -
\|f(0,z_2) \|^2_{H^2},\\ 
\|X_{2} g \|^2_{S^K_1} &\le
\| g \|^2_{S^K_1} - \|g(z_1,0) \|^2_{H^2}  . \end{align*}
 \end{prop}

\begin{proof} Let $ (K_1,K_2)$ be a pair of Agler kernels of $\phi$. Solving (\ref{eqn3.1}) for $K_1$ yields
 \begin{align}
\label{eqn3.2} K_1 = \frac{1 + z_1 \bar{w}_1K_2 }{1 - z_2 \bar{w}_2}
-\frac{\phi(z)\overline{\phi(w)} + K_2 }{1 - z_2 \bar{w}_2}. \end{align}
 Since the
left-hand-side of (\ref{eqn3.2}) is a positive kernel, it follows from Corollary \ref{corA.1} that 
\begin{align}
\label{eqn3.2c} \mathcal{H} \left(\frac{\phi(z)\overline{\phi(w)} + K_2 }{1 - z_2 \bar{w}_2}
\right ) \subseteq \mathcal{H} \left ( \frac{1 + z_1 \bar{w}_1K_2 }{1 -
z_2 \bar{w}_2}\right),\end{align}
and the containment is contractive. 
Let $f \in S^K_2.$ By Theorem \ref{thmA.3}, 
\begin{equation*} f \in \mathcal{H}
\left (\frac{\phi(z)\overline{\phi(w}) + K_2 }{1 - z_2 \bar{w}_2} \right). \end{equation*} 
Define the Hilbert space
\begin{equation*} Z_1 S^K_2: = \mathcal{H} \left( \frac{z_1 \bar{w}_1K_2 }{1 - z_2
\bar{w}_2}\right). \end{equation*}
Then $Z_1 S^K_2$ consists precisely of functions of the form $Z_1 g$ for
$g \in S^K_2$ and has inner product given by 
\begin{equation*} \langle Z_1 g_1, Z_1 g_2
\rangle_{Z_1S^K_2} = \langle g_1 ,g_2 \rangle_{S^K_2},\end{equation*}
for $Z_1g_1, Z_1g_2 \in
Z_1S^K_2.$ Then, (\ref{eqn3.2c}) paired with Theorem \ref{thmA.3} guarantees that we can
write 
\begin{align*} f(z)
= f_1(z_2) + z_1f_2(z),\end{align*}
for  $f_1 \in H^2_{2}$ and $f_2 \in S^K_2.$  
It clear that $f_1(z_2) = f(0,z_2)$, which means $f_2$ must equal $X_{1}f .$ 
Thus, $X_{1}f \in S^K_2$. As the containment in (\ref{eqn3.2c}) is
contractive and the decomposition of $f$ into $f_1$ and $Z_1f_2$ is unique, it follows from
Theorem \ref{thmA.3} that 
\begin{align*} \| f \|^2_{S^K_2} & \ge  \| f_1 \|^2_{H^2} + \| Z_1 f_2
\|^2_{Z_1 S^K_2} \\ & = \| f_1 \|^2_{H^2} + \| f_2 \|^2_{S^K_2} \\ &= \| f(0,z_2)
\|^2_{H^2} + \| X_{1} f \|^2_{S^K_2}. \end{align*}
 Analogous arguments give the result for $S^K_1$.
\end{proof}

The interested reader should also see \cite{bb11}, where Ball-Bolotnikov discuss the Gleason problem on $\mathcal{H}(K_1) \oplus \mathcal{H}(K_2).$ 
Now, for $\phi$ inner, it follows from Remark \ref{rem2.1} that if $(K_1,K_2)$ are Agler kernels of
$\phi$ such that 
\begin{equation*} \mathcal{H}_{\phi} =S^K_1 \oplus S^K_2, \end{equation*} 
then $ S^{min}_r \subseteq S^K_r$ for $r =1,2.$
When $\phi$ is a general Schur function, there are similar minimal sets. For $r=1,2,$ and  a holomorphic function $\psi$ on $\mathbb{D}^2$, define the set
\begin{equation*} \psi H^2_{r} := \big \{ \psi g:  g\in H^2_{r} \big \}.  \end{equation*}
Then we have the following:

\begin{prop} \label{prop3.2} Let $\phi \in H^{\infty}_1,$ and let $(K_1, K_2)$ be Agler kernels of
$\phi$. Then \begin{equation*} (X_{1} \phi ) H^2_{2} \subseteq S^K_2 \text{ and } (X_{2}
\phi ) H^2_{1} \subseteq S^K_1.\end{equation*} \end{prop}

\begin{proof} From the proof of Proposition \ref{prop3.1} and from Theorem \ref{thmA.3}, have 
the following set relationships: 
\begin{equation*} \phi
H^2_{2} = \mathcal{H} \left(\frac{\phi(z)\overline{\phi(w)}}{1 - z_2 \bar{w}_2}
\right ) \subseteq \mathcal{H} \left( \frac{1 + z_1 \bar{w}_1K_2
}{1 - z_2 \bar{w}_2} \right ).\end{equation*} 
Let $g \in H^2_2,$ so that $f= \phi g \in \phi H^2_{2}.$ As in Proposition
\ref{prop3.1}, we can write
 \begin{align*} f(z) = f_1(z_2) + z_1f_2(z),\end{align*}
 for $f_1 \in H^2_{2}$ and  $f_2 \in S^K_2.$  As $f_1$ must equal $f(0,z_2)$, it
follows that $f_2$ equals $X_{1}f$. Thus, $X_{1}f \in S^K_2.$  As
\begin{equation*} \big(X_{1}f \big)(z) =  \big(X_{1}\phi \big)(z) g(z_2),\end{equation*} the desired inclusion follows.
 Analogous arguments
give the result for $S^K_1$. \end{proof}

\begin{rem} \label{rem3.2} The arguments in Propositions \ref{prop3.1} and \ref{prop3.2} generalize to the
case where $\phi$ is in the Schur-Agler class of $\mathbb{D}^d$. Given positive holomorphic kernels $(K_1, \dots, K_d)$ such that
\begin{equation*}1-\phi(z)\overline{\phi(w)} = (1-z_1 \bar{w}_1)K_1+ \dots + (1-z_d\bar{w}_d)K_d,\end{equation*} and $r \in \{ 1, \dots, d \}$, it is easy to show that
\begin{align*}
&(1)&&  \mathcal{H} \left( \dfrac{K_r(z,w)}{ \prod_{j \ne r} (1-
z_j\bar{w}_j)}\right) \text{ is invariant under } X_{r}.  \\
&(2)&&  X_{r}f \in\mathcal{H} \left( \dfrac{K_r(z,w)}{ \prod_{j \ne r} (1-
z_j\bar{w}_j)}\right) \text{ for all } f \in \mathcal{H} \left( \frac{
\phi(z)\overline{\phi(w)}}{\prod_{j \ne r}(1 - z_j \bar{w}_j)} \right). 
\end{align*} These results look slightly different from Propositions \ref{prop3.1} and
\ref{prop3.2} because on $\mathbb{D}^d,$ it makes sense to number the kernels differently. \end{rem}

Recall that the Agler decompositions constructed in Section 2 for inner functions were obtained
via an orthogonal decomposition \begin{equation*} \mathcal{H}_{\phi} = S_1 \oplus S_2,\end{equation*} where $Z_r S_r \subseteq S_r$
and $\| M_{Z_r} \|_{S_r} \le 1$ for $r=1,2.$ It thus makes sense  to ask:
\begin{center}``For which Schur
functions $\phi$ does there exist such an orthogonal decomposition of $\mathcal{H}_{\phi}$?" \end{center}
Such orthogonal
decompositions will yield Agler decompositions as in the proof of Theorem \ref{thm2.1}.
The previous propositions allow us to characterize such Schur functions. For $\phi \in H^{\infty}_1,$ define: \begin{align}
\label{min1} V_1 &:= \Big \{ \sum_{m=1}^M (X^m_{2} \phi)(z)  f_m(z_1) :M \in \mathbb{N}, \ f_m \in H^2(\mathbb{D})\Big \}, \\ 
\label{min2} V_2 &:= \Big \{ \sum_{m=1}^M
(X^m_{1}\phi )(z) f_m(z_2) : \ M \in \mathbb{N}, \ f_m \in H^2(\mathbb{D})
\Big \}, \end{align} and define the closed subspaces $S^{min}_1:= clos_{\mathcal{H}_{\phi}} V_1$ and
$S^{min}_2:= clos_{\mathcal{H}_{\phi}} V_2.$ It follows from the proof of Lemma \ref{lem2.1}
that for $\phi$ inner, this definition of
$S^{min}_r$ is equivalent to the one given in Section 2. Observe that Propositions
\ref{prop3.1} and \ref{prop3.2} imply that each $V_r \subseteq S^K_r$ for any Agler spaces $(S^K_1,S^K_2)$ of $\phi$.

\begin{thm} \label{thm3.4} Let $\phi \in H^{\infty}_1.$ Then $\mathcal{H}_{\phi}$ has an orthogonal decomposition
\begin{equation*}\mathcal{H}_{\phi} = S_1 \oplus S_2,\end{equation*} 
into closed subspaces $S_1$ and $S_2$ such that $Z_r S_r \subseteq S_r$ and $\| M_{Z_r} \|_{S_r} \le
1$ for $r=1,2$ if and only if $S^{min}_1 \perp S^{min}_2$ in $\mathcal{H}_{\phi}.$ \end{thm}

\begin{proof} ($\Rightarrow$) Assume such an orthogonal decomposition of $\mathcal{H}_{\phi}$ exists. Using arguments similar to those in the
proof of Theorem \ref{thm2.1}, one can show that there are Agler kernels $(K_1,K_2)$ of $\phi$ with
$S^K_1 =S_1$ and $S^K_2=S_2.$ From Propositions \ref{prop3.1} and \ref{prop3.2}, each $V_r \subseteq
S_r$. As each $S_r$ is closed in $\mathcal{H}_{\phi}$, it is clear that $S^{min}_r \subseteq S_r.$ Since $S_1
\perp S_2$ in $\mathcal{H}_{\phi}$, we get $S^{min}_1 \perp S^{min}_2$ in $\mathcal{H}_{\phi}.$ \bigskip

\noindent
($\Leftarrow$) Assume $S^{min}_1 \perp S^{min}_2.$ Define $S^{max}_2:= \mathcal{H}_{\phi} \ominus S^{min}_1.$
We will show $S^{min}_1$ and $S^{max}_2$ have the desired properties. First, for a fixed $w \in \mathbb{D}^2$, write the  kernel $K_{\phi}(z,w)$ from (\ref{eqn:ker})
as $K_{\phi, w}(z)$.  Applying the backward shift $X_{1}$ to $K_{\phi, w }$ yields:
\begin{equation*}
(X_{1}K_{\phi, w})(z) = \bar{w}_1K_{\phi, w}(z) -
\overline{\phi(w)} \frac{(X_{1}\phi)(z) }{1-z_2\bar{w}_2}.\end{equation*} 
Now, we can
calculate the adjoint of $X_{1}$ in $\mathcal{H}_{\phi}$, which we denote by $X^*_{1}.$ Let
$f \in \mathcal{H}_{\phi}$ and $w \in \mathbb{D}^2$. Then
 \begin{align*} (X^*_{1}f)(w) &= \langle X_{1}^*f,
K_{\phi, w} \rangle_{\mathcal{H}_{\phi}} \\ 
&= \langle f, X_{1} K_{\phi, w}
\rangle_{\mathcal{H}_{\phi}} \\ 
 &= \langle f, \bar{w}_1 K_{\phi, w} -
\tfrac{\overline{\phi(w)}}{1-Z_2 \bar{w}_2} X_{1} \phi \rangle_{\mathcal{H}_{\phi}} \\ 
& = w_1 f(w) - \langle f, \tfrac{ X_{1} \phi}{1-Z_2 \bar{w}_2} \rangle_{\mathcal{H}_{\phi}}
\phi(w). \end{align*} 
Similarly, 
\begin{align*} (X^*_{2}f)(w) = w_2 f(w) - \langle f, \tfrac{
X_{2} \phi}{1-Z_1 \bar{w}_1} \rangle_{\mathcal{H}_{\phi}} \phi(w).\end{align*} 
Observe that \begin{equation*}\frac{
X_{2} \phi}{1-Z_1 \bar{w}_1} \in S^{min}_1 \text{ and } \frac{ X_{1}
\phi}{1-Z_2 \bar{w}_2} \in S^{min}_2, \end{equation*} for each $w \in \mathbb{D}^2.$ Then, for $f \in S^{min}_1$ and $g
\in S^{max}_2$, the orthogonality assumptions imply that
 \begin{align} \label{eqn3.4.1}
(X^*_{1}f)(z) &= z_1 f(z), \\  
 (X^{*}_{2}g)(z) &= z_2 g(z).\end{align} Let $f \in V_1.$ Then $Z_1f \in V_1$. As $S^{min}_1$ is a closed subspace of $\mathcal{H}_{\phi}$, we can use
(\ref{eqn3.4.1}) and Theorem \ref{lem3.1.a} to calculate 
\begin{align*} \| Z_1 f \|_{S^{min}_1} &=  \| Z_1 f \|_{\mathcal{H}_{\phi}} \\
& = \| X^{*}_{1} f \|_{\mathcal{H}_{\phi}} \\ 
&\le \| X_{1} \|_{\mathcal{H}_{\phi}} \| f \|_{\mathcal{H}_{\phi}} \\
& \le  \|f \|_{S^{min}_1}.\end{align*}
 Assume $\{f_n\} \rightarrow f$ in
$\mathcal{H}_{\phi}$, where $\{f_n\} \subseteq V_1.$ Then, as $\{Z_1 f_n\}$ satisfies 
\begin{equation*} \| Z_1
f_n - Z_1 f_m \|_{S^{min}_1} \le \| f_n - f_m \|_{S^{min}_1},\end{equation*} 
for $m,n \in \mathbb{N},$ the sequence $\{Z_1 f_n \}$ is Cauchy in $S^{min}_1.$
 Thus, $\{Z_1 f_n \}$ converges in
$S^{min}_1$ and in $H^2.$ As the limit in $H^2$ must be $Z_1f$,  the sequence must converges to $Z_1 f$ in $S^{min}_1$ as well, and   
\begin{equation*}\| Z_1 f \|_{S^{min}_1} \le \|f \|_{S^{min}_1}.\end{equation*} 
Thus, $Z_1 S^{min}_1
\subseteq S^{min}_1,$ and $\|M_{Z_1} \|_{S^{min}_1} \le 1.$ \bigskip

\noindent
Now consider $S^{max}_2.$ Let $g \in S^{max}_2.$ By the formula for $X^{*}_{2}$, we know
$Z_2 g = X^{*}_{2}g \in \mathcal{H}_{\phi}.$ Let 
\begin{equation*}f(z) = \sum_{m=1}^M (X^m_{2}
\phi )(z)  f_m(z_1)\end{equation*} 
be an arbitrary element in $V_1.$ It is clear that
$X_{2}f \in V_1$ as well. Then we can calculate
\begin{align*}  \langle Z_2 g, f \rangle_{\mathcal{H}_{\phi}} &=\langle X^*_{2} g, f
\rangle_{\mathcal{H}_{\phi}} \\ 
 &= \langle g , X_{2}f \rangle_{\mathcal{H}_{\phi}} \\
& = 0.\end{align*} 
As $f$ was arbitrary, $Z_2 g
\perp V_1$. It is immediate that $Z_2 g \perp S^{min}_1,$ and so $ Z_2 g \in S^{max}_2.$
Thus, $S^{max}_2$ is invariant under $M_{Z_2},$ and for $g \in S^{max}_2$, we
have 
\begin{align*}\| Z_2 g \|_{S^{max}_2}  &= \| X^*_{2} g
\|_{\mathcal{H}_{\phi}}  \\
& \le \| X_{2} \|_{\mathcal{H}_{\phi}} \| g \|_{\mathcal{H}_{\phi}}\\
&  \le  \|g \|_{S^{max}_2}.\end{align*} 
Thus, $\|M_{Z_2} \|_{S^{max}_2} \le 1,$ and the theorem is proved. \end{proof}

We will provide several examples to illustrate both the uses and limitations of Theorem
\ref{thm3.4}, but first we need an alternate definition of $\mathcal{H}_{\phi}.$ If $A: \mathcal{H}_1 \rightarrow \mathcal{H}_2$ is
a bounded operator between two Hilbert spaces, let $\mathcal{M}(A)$ denote the range of $A$ with
inner product defined by
\begin{equation*} \langle Ax, Ay \rangle_{\mathcal{M}(A)} = \langle x,y \rangle_{\mathcal{H}_1},\end{equation*} 
for all $x,y
\in \mathcal{H}_1$ orthogonal to the kernel of $A$. It is well-known and discussed at length in \cite{sar94}
that if $\phi \in H^{\infty}_1(\mathbb{D}),$ then 
\begin{equation*} \mathcal{H}_{\phi} = \mathcal{M}\big(
(1-T_{\phi}T_{\bar{\phi}})^{\frac{1}{2}}\big),\end{equation*}
 where $T_{\phi}$ is the Toeplitz operator with
symbol $\phi.$ The analysis generalizes immediately for $\phi \in H^{\infty}_1(\mathbb{D}^2)$. Define
$\mathcal{H}_{\bar{\phi}}$ to be $\mathcal{M}\big( (1-T_{\bar{\phi}}T_{\phi})^{\frac{1}{2}}\big),$ and observe
that $\mathcal{H}_{\bar{\phi}}$ is trivial for $\phi$ inner. Moreover, it follows from (I-8) in \cite{sar94}
that $f \in \mathcal{H}_{\phi}$ if and only if $T_{\bar{\phi}}f \in \mathcal{H}_{\bar{\phi}},$ and 
 for all  $f,g \in \mathcal{H}_{\phi},$
\begin{equation*} \langle f,g
\rangle_{\mathcal{H}_{\phi}} = \langle f,g \rangle_{H^2} + \langle T_{\bar{\phi}}f, T_{\bar{\phi}}g
\rangle_{\mathcal{H}_{\bar{\phi}}}.\end{equation*}

\begin{ex} Let $\phi$ be inner and consider $\psi:=t \phi$, where $0<t<1.$ Then, the $V_1$ and
$V_2$ spaces for $\phi$ and $\psi$ are identical. Let $f_r \in V_r$ for $r=1,2.$ As
$T_{\bar{\phi}}f_r =0,$ we have $T_{\bar{\psi}}f_r =0$ for each $r$. By our previous results about
inner functions, \begin{align*} \langle f_1,f_2 \rangle_{H^2} = \langle f_1, f_2 \rangle_{\mathcal{H}_{\phi}} =0, \end{align*}
which immediately implies 
\begin{equation*} \langle f_1, f_2 \rangle_{\mathcal{H}_{\psi}} = \langle f_1,f_2 \rangle_{H^2} +
\langle T_{\bar{\psi}}f_1, T_{\bar{\psi}}f_2 \rangle_{\mathcal{H}_{\bar{\psi}}} = 0.\end{equation*} 
Since $V_1 \perp V_2$
in $\mathcal{H}_{\psi},$ we get $S^{min}_1 \perp S^{min}_2$ in $\mathcal{H}_{\psi}.$ Theorem \ref{thm3.4} then implies
that there is an orthogonal decomposition of $\mathcal{H}_{\psi}$ yielding an Agler decomposition of $\psi.$
\end{ex}

Not all examples arise from inner functions or one-variable functions.

\begin{ex} Consider $\phi(z) =\tfrac{1}{2} (z_1 + z_1z_2).$ Then, we  can calculate 
\begin{align*} V_1
&= \{ z_1f(z_1): f\in H^2(\mathbb{D}) \} \\
V_2 &= \{
(1+z_2)f(z_2): f \in H^2(\mathbb{D})\}.\end{align*} 
Moreover, for every $f \in V_2$, we have 
\begin{equation*}
T_{\bar{\phi}}f = P_{H^2} \big( \tfrac{1}{2}\bar{z}_1(1+\bar{z}_2)f(z_2) \big)
=0.\end{equation*} As $V_1 \perp V_2$ in $H^2$, for any $f_1 \in V_1$, $f_2 \in V_2$, we
have
 \begin{equation*} \langle f_1, f_2 \rangle_{\mathcal{H}_{\phi}} = \langle f_1,f_2 \rangle_{H^2} + \langle
T_{\bar{\phi}}f_1, T_{\bar{\phi}}f_2 \rangle_{\mathcal{H}_{\bar{\phi}}} = 0.\end{equation*} 
Thus, $V_1 \perp V_2$ in
$\mathcal{H}_{\phi}$ and so $S^{min}_1 \perp S^{min}_2$ in $\mathcal{H}_{\phi}.$ This same argument holds for any $\phi$
such that $V_1 \perp V_2$ in $H^2,$ and $T_{\bar{\phi}}V_r =\{0\}$ for $r=1$ or $r=2.$ \end{ex}

It is also quite easy to find functions for which the assumptions of Theorem \ref{thm3.4} fail. 

\begin{ex} Set $\phi(z) = \frac{1}{2}(z_1 + z_2).$ 
Then, for $r=1,2,$ the set $V_r$ contains precisely the functions in $H^2_{r}.$ 
As $1 \in V_1 \cap V_2,$ we cannot have
$V_1 \perp V_2$ in $\mathcal{H}_{\phi},$ since
\begin{equation*} \|1 \|^2_{\mathcal{H}_{\phi}}  = \|1 \|^2_{H^2} + \|T_{\bar{\phi}} 1 \|^2_{\mathcal{H}_{\bar{\phi}}} =1. \end{equation*} \end{ex}
 
\section{Rational Inner Functions} 

We first review the structure of rational inner functions on the bidisk.

\begin{defn} A set $X \subseteq \mathbb{C}^2$ is \emph{determining} for an algebraic set $A\subseteq \mathbb{C}^2$ if
$f \equiv 0$ whenever $f$ is holomorphic on $A$ and $f|_{X \cap A} =0$. A polynomial $p$ is
\emph{atoral} if $\mathbb{T}^2$ is not determining for any of the irreducible components of the zero set
of $p$. \end{defn}

For more information about determining sets and toral/atoral polynomials see \cite{ams06}. Let $p$
be a polynomial on $\mathbb{C}^2$ with $\deg_{r}p=j_r$ for $r=1,2$, and define its reflection $\tilde{p}$
as 
\begin{equation*}\tilde{p}(z) := z_1^{j_1}z_2^{j_2} \overline{ p\big(
\tfrac{1}{\bar{z}} \big) }.\end{equation*}

\begin{rem} \label{rem4.1} Let $\phi \in H^{\infty}_1$ be rational inner. By the atoral-toral
factorization of Agler-McCarthy-Stankus in \cite{ams06}, there are unique (up to a unimodular
scalar) functions $m$ and $p$ such that  
\begin{align} \label{eqn4.1} \phi(z) = m(z)
\frac{\tilde{p}(z)}{p(z)},\end{align} 
$m$ is a monomial, and $p$ is an atoral polynomial
with no zeros in $\mathbb{D}^2$ and finitely many zeros on $\mathbb{T}^2.$ Then, $\deg\phi = (k_1,k_2),$ where $ k_r = \deg_{r} m +
\deg_{r} p$ for $r=1,2.$  Also, every function of form $(\ref{eqn4.1})$ is rational inner. \end{rem}

Recall that $S^{max}_r$ can be viewed equivalently as
a space of holomorphic functions on $\mathbb{D}^2$ contained in $H^2$ and a space of $L^2$ functions contained in $(\ref{eqn:H2})$. 
Then, the following result about $S^{max}_r$ for $r=1,2$ can be viewed as both a statement about the analytic functions and a 
statement about their boundary values.

\begin{lem} \label{lem4.1} Let $\phi \in H^{\infty}_1$ be rational inner with representation $(\ref{eqn4.1}).$ Then
 \begin{align*}
S^{max}_1 &\subseteq \Big \{ \tfrac{f}{p} \in H^2: f \in H^2 \text{ and } \hat{f}(n_1,n_2) =0 \text{ for } n_2 \ge k_2 
\Big \}, \\ 
S^{max}_2 &\subseteq \Big \{ \tfrac{f}{p} \in
H^2 :f \in H^2 \text{ and } \hat{f}(n_1,n_2) =0 \text{ for } n_1 \ge k_1  \Big
\}. \end{align*} 
\end{lem}

\begin{proof} Let $g \in S^{max}_1.$ By Lemma \ref{lem2.1}, there is an $h \in L^2_{*-}$ such that $g =
\tfrac{m \tilde{p}}{p}h.$ Then
\begin{equation*} m\tilde{p}h = pg \in H^2. \end{equation*} 
Since $h \in
L^2_{*-}$ and $\deg_{2}(m \tilde{p}) =
k_2,$ if we set $f:=m\tilde{p}h$, it can be shown that $\hat{f}(n_1,n_2) =0$ 
whenever $n_2 \ge k_2.$ The result follows similarly for $S^{max}_2.$ \end{proof}

The following finiteness result is due to Cole-Wermer \cite[Corollary 2.2]{colwer99},
and the specific dimension bounds were shown by Knese in \cite[Theorem 2.10]{kn10a}. In
\cite{bsv05}, Ball-Sadosky-Vinnikov gave an alternate proof of the Cole-Wermer result for a subset of the
Agler kernels of $\phi$.

We use the $S^{max}_r$ subspaces to provide a very simple proof of the Cole-Wermer
result, which is distinct from the arguments in \cite{bsv05}. In this proof, 
we index Taylor coefficients by $m,n$ instead of $n_1,n_2$ to simplify notation.

\begin{thm} \label{thm4.1} Let $\phi \in H^{\infty}_1$ be rational inner with representation $(\ref{eqn4.1}),$ and let $(K_1,
K_2)$ be Agler kernels of $\phi$. Then, 
\begin{equation*} \text{dim} (\mathcal{H}(K_1)) \le k_2(k_1+1) \text{ and } \text{dim}(\mathcal{H}(K_2)) \le
k_1(k_2+1) . \end{equation*}
 Setting $m_1:= \text{dim}(\mathcal{H}(K_1))$ and $m_2:= \text{dim} (\mathcal{H}(K_2))$, we can write
\begin{equation*}K_1(z, w) = \frac{1}{p(z) \overline{p(w)}} \sum_{i=1}^{m_1}q_i(z)
\overline{q_i(w)}  \text{ and }  K_2(z, w)= \frac{1}{p(z)
\overline{p(w)}} \sum_{j=1}^{m_2} r_j(z) \overline{r_j(w)}, \end{equation*} 
for polynomials $\{q_i \}$ with $\deg q_i \le (k_1,k_2-1)$ for $1 \le i \le m_1$, and polynomials $\{r_j\}$ with $\deg r_j \le (k_1-1, k_2)$  for 
$1 \le j \le m_2.$ \end{thm}
 
\begin{proof} Let $(K_1,K_2)$ be Agler kernels of $\phi$. For a fixed $w \in \mathbb{D}^2$ and $r=1, 2$, it follows from basic properties of
reproducing kernel Hilbert spaces that $K_r(\cdot,w) \in S^K_r \subseteq S^{max}_r$. By  Lemma \ref{lem4.1}, we can write 
\begin{align} K_1(z, w) &= \dfrac{1}{p(z)} \sum_{\substack{m \ge 0 \\ 0 \le n <k_2}} a_{mn}(
w) z_1^{m} z_2^{n}, \label{eqn4.3}\\ K_2(z, w) &= \dfrac{1}{p(z)}
\sum_{\substack{0 \le m < k_1 \\ n \ge 0}} b_{mn}( w) z_1^{m} z_2^{n}, \label{eqn4.4}
\end{align} 
for $z \in \mathbb{D}^2$ and coefficients $a_{mn}(w)$ and $b_{mn}(w)$ in $l^2(\mathbb{N}^2).$ Substituting
(\ref{eqn4.1}), (\ref{eqn4.3}), and (\ref{eqn4.4}) into (\ref{eqn3.1}) and canceling the denominator
$p(z)$ yields 
\begin{align*} p(z) &- \overline{\phi(w)}(m\tilde{p})(z) \\
& =
(1-z_1\bar{w}_1) \sum_{\substack{0 \le m < k_1 \\ n \ge 0}}b_{mn}( w) z_1^{m}
z_2^{n} + (1-z_2\bar{w}_2) \sum_{\substack{m \ge 0 \\ 0 \le n <k_2}} a_{mn}( w)
z_1^{m} z_2^{n}.\end{align*} 
An examination of the degrees of the terms in the above equation shows that 
\begin{align*} K_1(z, w)& = \dfrac{1}{p(z)}
\sum_{\substack{ 0 \le m \le k_1 \\ 0 \le n <k_2}} a_{mn}( w) z_1^{m} z_2^{n}, \\ 
K_2(z, w)
&= \dfrac{1}{p(z)} \sum_{\substack{ 0 \le m < k_1 \\ 0 \le n  \le k_2}} b_{mn}( w) z_1^{m}
z_2^{n}. \end{align*}
By the canonical construction of $\mathcal{H}(K_1)$ from $K_1$, the linear span of the set of functions
$\{K_1(\cdot, w)\}_{w \in \mathbb{D}^2}$ is dense in $\mathcal{H}(K_1).$ Fix $g \in \mathcal{H}(K_1),$ and
let $\{f_n/ p \}$ be a sequence with elements in the linear span of $\{K_1(\cdot, w)\}_{w \in
\mathbb{D}^2}$ that converges to $g$. Then for each $n$, $\deg f_n \le (k_1, k_2-1).$ 
As $\{f_n/ p \}$ also converges to $g$ in $H^2,$ it follows that 
$\{f_n\}$ converges to $gp$ in $H^2$ and that $\deg gp \le (k_1, k_2-1).$ As $g= gp / p,$ it follows that 
\begin{equation*}\mathcal{H}(K_1) \subseteq \Big \{ \tfrac{f}{p}: f(z) = \sum_{\substack{0 \le m \le k_1 \\ 0 \le n <k_2}}c_{mn} z_1^{m} z_2^{n} \Big \}, \text{ and }
\text{dim}(\mathcal{H}(K_1)) \le k_2(k_1+1).\end{equation*} 
Let $m_1= \text{dim} (\mathcal{H}(K_1)),$ and let $\{f_i\}_{i=1}^{m_1}$ be an
orthonormal basis for $\mathcal{H}(K_1)$. For each $i$, we have $f_i$ = $\tfrac{q_i}{  p},$ where $\deg q_i
\le (k_1,k_2-1 ).$ By Theorem \ref{thmA.1}, 
\begin{align*} K_1 (z,w) = \frac{1}{p(z)
\overline{p(w)}} \sum_{i=1}^{m_1}q_i(z) \overline{q_i(w)}. \end{align*} 
An analogous argument
holds for $\mathcal{H}(K_2).$ \end{proof}

Given a rational inner $\phi$ with $\deg \phi = (k_1,k_2),$ one can actually choose $(K_1,K_2)$ so
that $\text{dim}(\mathcal{H}(K_1)) =k_2$ and $\text{dim}(\mathcal{H}(K_2)) =k_1$. Such decompositions are discussed by
Kummert in \cite{kum89a} and Knese in \cite{kn08ub}.

We now restrict attention to rational inner functions continuous on $\overline{\mathbb{D}^2}$. The following
results about $S^{max}_1$ and $S^{max}_2$ are proven by Ball-Sadosky-Vinnikov in
\cite[Proposition 6.9]{bsv05}. We also consider a related result for $\mathcal{H}_{\phi},$ which simplifies the proofs
for $S^{max}_1$ and $S^{max}_2$.

\begin{prop} \label{prop4.1} Let $\phi \in H^{\infty}_1$ be rational inner and continuous on $\overline{\mathbb{D}^2}$ with
representation $(\ref{eqn4.1})$. Then 
\begin{align*} \mathcal{H}_{\phi} &=\Big \{ \tfrac{f}{p} : f \in H^2 \text{ and } \hat{f}(n_1,n_2) =0 \text{ if } n_1 \ge k_1 \text{ and } n_2 \ge k_2 \Big \},\\ 
S^{max}_1 & =  \Big \{ \tfrac{f}{p} :f \in H^2 \text{ and } \hat{f}(n_1,n_2) =0 \text{ if }  n_2 \ge k_2 \Big \}, \\
S^{max}_2 & =  \Big \{ \tfrac{f}{p} : f \in H^2 \text{ and } \hat{f}(n_1,n_2) =0 \text{ if } n_1\ge k_1 \Big \}. \end{align*} 
\end{prop}
 \begin{proof} Because $\phi$ is continuous on
$\overline{\mathbb{D}^2}$, we have $p, \frac{1}{p} \in H^{\infty}$ and so, $\frac{1}{p} H^2 =H^2$. Set
 \begin{equation*}q(z):= \overline{ p\big(
\tfrac{1}{\bar{z}} \big) }.\end{equation*} 
It is clear that $q, \frac{1}{q} \in
L^{\infty}$ and so, these functions multiply $L^2$ into $L^2.$ Let $f \in H^2$ and $g \in L^2 \ominus H^2$. As $q \equiv \bar{p}$ on $\mathbb{T}^2$, we have
\begin{align*} \langle qg, f \rangle_{L^2} &= \langle g, pf \rangle_{L^2}=0, \\
 \langle \tfrac{1}{q}g, f \rangle_{L^2} &= \langle g, \tfrac{1}{p} f \rangle_{L^2}=0.
 \end{align*}
 Then, it is immediate that 
\begin{equation*}q [ L^2 \ominus H^2]\subseteq L^2 \ominus H^2 \text{ and }
\tfrac{1}{q}[ L^2 \ominus H^2] \subseteq L^2 \ominus H^2.\end{equation*}
 Thus, $q [ L^2 \ominus H^2]=L^2 \ominus
H^2 .$ By the characterization of $\mathcal{H}_{\phi}$ in Remark \ref{rem2.1.b}, we have
 \begin{align*} \mathcal{H}_{\phi} &=
 \phi  [ L^2 \ominus H^2] \cap H^2 \\ &= \tfrac{1}{p} \Big [ m \tilde{p} [L^2 \ominus
H^2] \cap H^2 \Big ] \\ &= \tfrac{1}{p} \Big [ z_1^{k_1} z_2^{k_2} q [L^2 \ominus H^2]
\cap H^2 \Big ]\\ &= \tfrac{1}{p} \Big [ z_1^{k_1} z_2^{k_2} [L^2 \ominus H^2] \cap H^2
\Big ]\\ &= \Big \{ \tfrac{f}{p} : f \in H^2 \text{ and } \hat{f}(n_1,n_2) =0 \text{ if } n_1\ge k_1 \text{ and } n_2 \ge k_2 \Big \}, \end{align*} 
as desired. We now prove the result for $S^{max}_1$. Set 
\begin{equation*}S_1 :=  \Big \{ \tfrac{f}{p} :f \in H^2 \text{ and } \hat{f}(n_1,n_2) =0 \text{ if }  n_2 \ge k_2 \Big \}.\end{equation*} 
From Lemma \ref{lem4.1}, we know $S^{max}_1 \subseteq S_1.$ Moreover, $S_1$ is
invariant under $M_{Z_1}$ and by the characterization of $\mathcal{H}_{\phi}$, we have $S_1 \subseteq
\mathcal{H}_{\phi}$. By the definition of $S^{max}_1,$ we have $S_1 \subseteq S^{max}_1$ and so, the two sets
are equal. The result follows similarly for $S^{max}_2.$\end{proof}

The following corollary is a special case of Theorem \ref{thm2.2}. This result was originally proved by
Knese in \cite[Corollary 1.16]{kn08ub}.

\begin{cor} \label{cor4.2} Let $\phi \in H^{\infty}_1$ be rational inner and continuous on $\overline{\mathbb{D}^2}$ with
representation $(\ref{eqn4.1})$. Then $\phi$ has a unique Agler decomposition if and only if $\phi$
is a function of one variable. \end{cor} 
\begin{proof} By Proposition \ref{prop4.1},
 \begin{align} \label{eqn4.6}
S^{max}_1 \cap S^{max}_2 = \Big \{ \tfrac{f}{p} :f \in H^2 \text{ and } \hat{f}(n_1,n_2) =0 \text{ if } n_1 \ge k_1  \text{ or }  n_2 \ge k_2 \Big \}. \end{align}
 As $S^{max}_1 \cap S^{max}_2 = H^2 \cap
\phi L^2_{--},$ it follows from Theorem \ref{thm2.2} that $\phi$ has a unique Agler decomposition
if and only if $(\ref{eqn4.6}) = \{0\},$ which occurs if and only if $k_1$ or $k_2$ is zero. \end{proof}

Corollary \ref{cor4.2} does not hold for general rational inner functions. Rather, we can construct
rational inner functions with arbitrarily high degree and unique Agler decompositions.

\begin{prop} \label{thm4.2} Let $(k_1,k_2) \in \mathbb{N}^2.$ Then there exists a rational inner function $\phi$
such that $\deg \phi = (k_1,k_2),$ and $\phi$ has a unique Agler decomposition. \end{prop}

\begin{proof} Let $(k_1,k_2) \in \mathbb{N}^2.$ By Theorem \ref{thm2.2}, an inner function $\phi$ has a unique Agler
decomposition if and only if $H^2 \cap \phi L^2_{--} = \{0\}.$ Let $p$ be an atoral polynomial with $\deg p = (k_1,k_2)$ and with no zeros on $\mathbb{D}^2.$ Then, $\phi = \tfrac{\tilde{p}}{p}$ is
rational inner with $\deg \phi = (k_1,k_2)$.  As $S^{max}_1 \cap S^{max}_2 = H^2 \cap
\phi L^2_{--},$  we can use Lemma \ref{lem4.1} to conclude that 
\begin{align} H^2 \cap \phi L^2_{--} &
\subseteq \Big \{  \tfrac{q}{p} \in H^2: q \in H^2 \text{ and } \hat{q}(n_1,n_2)=0 \text{ if } n_1 \ge k_1 \text{ or } n_2 \ge k_2  \Big \}. \label{eqn4.5.1}
\end{align} 
Let $L$ denote the set on the right-hand-side of $(\ref{eqn4.5.1}).$ We will construct a $\phi$ such that
$L$ is trivial. Let $p$ be  an atoral polynomial with $\deg p = (k_1,k_2)$ and with zeros at the following
$k_1^{th}$ and $k_2^{th}$ roots of unity: 
\begin{align} \label{eqnroot} \Big (e^{ \frac{2\pi i k}{k_1}},
e^{ \frac{2\pi i j}{k_2}} \Big ),  \end{align}
where $1 \le k \le k_1, \ 1 \le j \le k_2.$ In particular, we take $p(z)= 3 - z_1^{k_1}-z_2^{k_2} -z_1^{k_1}z_2^{k_2}$
and will consider $\phi = \frac{ \tilde{p}}{p}.$
Using the power series representation of $p$ centered at each root of unity, one can use basic
estimates to show 
\begin{align*} \tfrac{1}{|p|^2} \text{ is not integrable near each }\Big (e^{ \frac{2\pi
ik}{k_1}}, e^{ \frac{2\pi ij}{k_2}} \Big). \end{align*} 
It follows that if there is a function
$q$ with $\tfrac{q}{p} \in H^2,$ then $q$ vanishes at each root of unity in (\ref{eqnroot}). Observe that, 
if $ \tfrac{q}{p} \in L,$ then $q$ is a polynomial with $\deg_{r} q < k_r,$ for $r=1,2.$  We can write
\begin{equation*} q(z) = \sum_{\substack{ 0 \le m < k_1 \\ 0 \le n <k_2}}
a_{mn} z_1^m z_2^n,  \text{ where }   q \Big ( e^{ \frac{2\pi i k}{k_1}}, e^{
\frac{2\pi ij}{k_2}} \Big )=0,\end{equation*} 
for all $k,j$ with $1 \le k \le k_1$ and $1 \le j \le k_2$. We will show that such a $q$ must satisfy $q \equiv 0.$ For each $k$,
where $1 \le k \le k_1,$ define 
\begin{equation*}q_k(z_2) := q \Big (e^{ \frac{2\pi ik}{k_1}},
z_2 \Big ) = \sum_{0 \le n <k_2} \Big( \sum_{ 0 \le m < k_1} a_{mn} e^{ \frac{2\pi ikm}{k_1}} \Big) z_2^n.\end{equation*} 
As $\deg q_k \le k_2-1$ and $q_k$ has $k_2$ zeros, $q_k \equiv 0.$
That implies 
\begin{align} \label{eqn4.6.1} \sum_{ 0 \le m < k_1} a_{mn} e^{ \frac{2\pi i km}{k_1}} = 0,
\end{align} 
 for all $k$ and $n$ with $1 \le k \le k_1, \ 0 \le n \le k_2-1.$ Fix $n$ with $ 0 \le n \le
k_2-1$. It follows from (\ref{eqn4.6.1}) that we have the following matrix equation: 
\begin{equation*} \left[
\begin{array}{cccc} 1 & e^{ \frac{2\pi i}{k_1}} & \cdots & \Big (e^{ \frac{2\pi i}{k_1}}
\Big)^{k_1-1} \\ 1 & e^{ \frac{4\pi i}{k_1}} & \cdots & \Big(e^{ \frac{4\pi i }{k_1}}\Big)^{k_1-1} 
\\ \vdots & \vdots & & \vdots \\ 1 & e^{ \frac{2 k_1 \pi i}{k_1}} &
\cdots & \Big( e^{ \frac{2 k_1\pi i }{k_1}} \Big)^{k_1-1} \end{array} \right] \cdot \left[
\begin{array}{c} a_{0n} \\ a_{1n} \\ \vdots \\ \\ a_{(k_1-1)n} \end{array} \right] = \left[
\begin{array}{c} 0 \\ 0 \\ \vdots \\ \\ 0 \end{array} \right]. \end{equation*} 
Observe that the matrix is a
Vandermonde matrix. It then has determinant given by 
\begin{equation*}\prod_{1 \le s < t \le k_1} \Big (e^{ \frac{2\pi i
s }{k_1}} - e^{ \frac{2\pi i  t }{k_1}} \Big) \ne 0.\end{equation*} 
As the matrix is nonsingular, each
$a_{mn}=0,$ and so $q\equiv 0.$ Thus, 
\begin{equation*}\phi(z)= \frac{\tilde{p}(z)}{p(z)} =
\frac{3z_1^{k_1}z_2^{k_2} - z_1^{k_1}-z_2^{k_2}-1 }{3 -
z_1^{k_1}-z_2^{k_2} -z_1^{k_1}z_2^{k_2}}\end{equation*} 
has a trivial $L$ set and hence,
a unique Agler decomposition. \end{proof}

\section{Stable Polynomials}  We end with an application of the analysis
in Section 4.

Let $d \ge 2,$ and let $\phi$ be rational inner on $\mathbb{D}^d$ with $\deg \phi = (k_1, \dots, k_d).$ Again, by the analysis of Agler-McCarthy-Stankus in \cite{ams06}, 
$\phi$ has an (almost) unique representation as 
\begin{align} \label{eqn4.7}
\phi(z) = m(z) \frac{\tilde{p}(z)}{p(z)},\end{align} 
for a monomial $m$ and an
atoral polynomial $p$ with no zeros on $\mathbb{D}^d$, such that $\deg_{r} \phi = \deg_{r} m +\deg_{r} p$ for
each $r.$ Moreover, any function of the form (\ref{eqn4.7}) is rational inner. We also define the reproducing kernel Hilbert space 
\begin{equation*} \mathcal{H}_{\phi} := \mathcal{H} \left (
\frac{1-\phi(z)\overline{\phi(w)}}{\prod_{i=1}^d (1-z_i\bar{w}_i)} \right).\end{equation*}

For a fixed $d$, we define the notation $H^2:=H^2(\mathbb{D}^d)$ and 
$L^2:=L^2(\mathbb{T}^d).$ The arguments in Proposition \ref{prop4.1} generalize easily to yield the
following result:

\begin{prop} \label{prop5.1} Let $\phi \in H^{\infty}_1(\mathbb{D}^d)$ be rational inner and continuous on
$\overline{\mathbb{D}^d}$ with $\deg \phi = (k_1, \dots, k_d)$ and representation $(\ref{eqn4.7}).$ Then \begin{equation*} \mathcal{H}_{\phi}= \tfrac{1}{p} \Big[ z^{k_1}_1 \cdots z^{k_d}_d \big [ L^2 \ominus H^2 \big ] \cap
H^2 \Big ]. \end{equation*} \end{prop}

We say a polynomial $p$ in $d$ complex variables is \emph{stable} if $p$ has no zeros on
$\overline{\mathbb{D}^d}$. We can now generalize a result of Knese \cite[Theorem 1.1]{kn08bs} about
stable polynomials in two complex variables and simultaneously, provide a simple proof of the
original result.

\begin{thm} \label{thm4.3} Let $p$ be a non-constant polynomial in $d$ complex variables.
Then $p$ is stable if and only if there is a constant $c>0$ such that for all $z \in \mathbb{D}^d,$ 
\begin{align} \label{eqn5.1.2}
|p(z)|^{d}-|\tilde{p}(z)|^{d} \ge c \prod_{i=1}^d \big(1-|z_i|^2 \big).\end{align} 
\end{thm}

\begin{proof} ($\Rightarrow$) Assume $p$ is a stable polynomial in $d$ complex variables.
As $p$ has no zeros on $\overline{\mathbb{D}^d}$, 
then $p$ is atoral, and $\phi = \tfrac{\tilde{p}}{p}$ is inner. By Proposition \ref{prop5.1},
 \begin{equation*} \mathcal{H}_{\phi}=
\tfrac{1}{p} \Big[ z^{k_1}_1 \cdots z^{k_d}_d \big [ L^2 \ominus H^2 \big ] \cap H^2
\Big ].\end{equation*} 
It is immediate that $\tfrac{1}{p} \in \mathcal{H}_{\phi}$, and so there is a constant $c_1 >0$ such
that 
\begin{equation} \label{eqn:pker} \frac{1-\phi(z)\overline{\phi(w)}}{\prod_{i=1}^d (1-z_i\bar{w}_i)} -
\frac{c_1}{p(z)\overline{p(w)}}\end{equation} 
is a positive kernel. Setting $w = z$ in $(\ref{eqn:pker})$ gives a nonnegative expression and rearranging terms yields 
\begin{equation*} |p(z)|^{2}-|\tilde{p}(z)|^{2} \ge c_1 \prod_{i=1}^d
\big (1-|z_i|^2 \big).\end{equation*} 
As $p$ has no zeros on $\overline{\mathbb{D}^d}$ and since $p, \tilde{p}$ are clearly bounded on $\overline{\mathbb{D}^d},$ there is a constant $c_2>0$ such that 
\begin{align*} |p(z)| -|\tilde{p}(z)| & \ge  c_1  \frac{1}{|p(z)| + |\tilde{p}(z)|} \prod_{i=1}^d
\big (1-|z_i|^2 \big ) \\
& \ge c_2 \prod_{i=1}^d
\big (1-|z_i|^2 \big). \end{align*}
Again, as $p$ does not vanish on $\overline{\mathbb{D}^d}$, there is a constant $c_3>0$ such that 
\begin{align*} |p(z)|^d -|\tilde{p}(z)|^d & =  \big(  |p(z)| -|\tilde{p}(z)| \big) \bigg( \sum_{j=1}^d |p(z)|^{j-1} |\tilde{p}(z)|^{d-j} \bigg) \\
& \ge c_3 \big(  |p(z)| -|\tilde{p}(z)| \big)  \\
& \ge  c  \prod_{i=1}^d
\big (1-|z_i|^2 \big), \end{align*}
where $c=c_2c_3>0.$ \bigskip

\noindent
($\Leftarrow$) Assume $p$ satisfies equation $(\ref{eqn5.1.2})$. Proceeding towards a contradiction,
assume $p$ (and thus, $\tilde{p}$) has a zero on $\partial \mathbb{D}^d.$ Without loss of generality, we
can assume the zero occurs at a point $(\tau_1, \dots, \tau_d) \in
\mathbb{D}^{n_1} \times \mathbb{T}^{n_2},$ where $n_1 +n_2 = d.$ Assume $n_2 < d$. As $p(r\tau_1,\dots, r\tau_d) = O(1-r)$
and $\tilde{p}(r\tau_1,\dots, r\tau_d) = O(1-r)$, it is immediate that 
\begin{align} |p(r\tau_1,\dots, r\tau_d)|^{d} -
|\tilde{p}(r\tau_1,\dots, r\tau_d)|^d = O(1-r)^d. \label{eqn5.2}\end{align}
 Combining (\ref{eqn5.1.2}) and
(\ref{eqn5.2}) and using the fact that $n_2 <d,$ we obtain a contradiction as $r \nearrow 1.$ \medskip

\noindent
Assume $n_2=d.$ For some constant $a$, we have
$p(r\tau_1, \dots, r\tau_d) = a(1-r) +O(1-r)^2,$ and
 \begin{align*} \tilde{p}(r\tau_1, \dots, r\tau_d) &= r^{k_1 + \dots
+ k_d} \tau_1^{k_1} \cdots \tau_d^{k_d}\bar{ p}( \tfrac{\tau_1}{r}, \dots, \tfrac{\tau_d}{r}) \\  &=
r^{k_1 + \dots+ k_d} \tau_1^{k_1} \cdots \tau_d^{k_d} \big[ \bar{a}(1-\tfrac{1}{r}) +O(1-r)^2 \big]. \end{align*}

Using our equations for $p(r\tau_1, \dots, r\tau_d) $ and $\tilde{p}(r\tau_1, \dots, r\tau_d)$, we have
\begin{align} |p(r\tau_1, \dots & , r\tau_d)|^{d}  - |\tilde{p}(r\tau_1, \dots, r\tau_d)|^{d}  \nonumber \\
& =  \big|a(1-r) +O(1-r)^2 \big |^{d} - r^{d(k_1 + \dots + k_d)} \big | \bar{a}(1-\tfrac{1}{r})
+O(1-r)^2 \big |^{d} \nonumber\\ &=  |a|^{d}(1-r)^{d} \big[ 1 - r^{d(k_1 + \dots + k_d-1)}
\big] +O(1-r)^{d+1}\nonumber \\ &= O(1-r)^{d+1}. \label{eqn5.1.3} \end{align}
Combining (\ref{eqn5.1.2}) and (\ref{eqn5.1.3}), we get a contradiction as $r \nearrow 1.$ \end{proof}

\section*{Acknowledgments}
I would like to thank John McCarthy for his
guidance during this research and Greg Knese, Joseph Ball, and the referee for their many useful insights and 
suggestions.

\appendix \section{}

This section contains results about reproducing kernel Hilbert spaces that have been used in the
paper. For more information about reproducing kernels and their associated Hilbert spaces, 
see \cite{aro50}, \cite{bv03b}, Chapter 2
in \cite{alp01}, and Chapter 2 in \cite{ ampi}.

Let $\Omega$ be a set in $\mathbb{C}^d$. In this paper, we say a function $K: \Omega \times \Omega
\rightarrow \mathbb{C}$ is a positive kernel on $\Omega$ if, for all finite sets $\{\lambda^1, \dots, \lambda^m \} \subseteq \Omega$, the matrix
with entries $K(\lambda^i,\lambda^j)$ is positive semidefinite.

A reproducing kernel Hilbert space $\mathcal{H}$ on $\Omega$ is a Hilbert space of functions on $\Omega$ such that, for
each $w \in \Omega$, point evaluation at $w$ is a continuous linear functional. Thus, there exists an element $K_{w} \in 
\mathcal{H}$ such that for each $f \in \mathcal{H},$ 
\begin{equation*} \langle f, K_{w} \rangle_{\mathcal{H}} = f(w).\end{equation*} 
As $K_{w}(z) = \langle K_{w}, K_{z}\rangle_{\mathcal{H}}$, we can regard $K$ as a function on
$\Omega \times \Omega$ and write $K(z, w): = K_{w}(z).$ Such a $K$ is a positive kernel, and
the space $\mathcal{H}$ with reproducing kernel $K$ is denoted by $\mathcal{H}(K).$ If $\mathcal{H}(K)$ is a space of
holomorphic functions, then $K$ is holomorphic in $z$ and conjugate-holomorphic in $w$. The
following property follows from Parseval's identity:

\begin{thm} \label{thmA.1} \cite[Proposition 2.18]{ampi} Let $\mathcal{H}(K)$ be a reproducing kernel Hilbert
space on $\Omega$ and let $\{f_i\}_{i \in I}$ be an orthonormal basis for $\mathcal{H}(K)$. Then 
\begin{equation*}K(z,
w) = \sum_{i \in I} f_i(z) \overline{ f_i(w)}.\end{equation*} \end{thm}

The following result relates reproducing kernel Hilbert spaces and positive kernels. It was originally stated for positive kernels that do not vanish on the diagonal. However, the result remains true if we consider positive kernels that vanish at some diagonal points.

\begin{thm} \label{thmA.2} \cite[Theorem 2.23]{ampi} Given a positive kernel $K$ on $\Omega,$ there is a reproducing kernel Hilbert space $\mathcal{H}(K)$ on $\Omega$ with reproducing kernel $K$. \end{thm}

The following results allow us to construct new reproducing kernel Hilbert spaces from
known spaces.

\begin{thm} \label{thmA.3} \cite[Theorem 5]{bt04} Let $\mathcal{H}(K_1)$ and $\mathcal{H}(K_2)$ be reproducing kernel
Hilbert spaces on $\Omega.$ Then, $K=K_1+K_2$ is the reproducing kernel of $ \mathcal{H}(K_1) + \mathcal{H}(K_2)$ with
norm $\| \cdot \|_{\mathcal{H}(K)}$ defined by 
\begin{equation*} \| f\|_{\mathcal{H}(K)}^2 = \min_{\substack{f = f_1 +f _2 \\ f_1 \in
\mathcal{H}(K_1), \ f_2 \in \mathcal{H}(K_2)}} \|f_1\|_{\mathcal{H}(K_1)}^2 + \|f_2\|_{\mathcal{H}(K_2)}^2.\end{equation*} \end{thm}

\begin{thm} \label{thmA.4} \cite[Theorem 11]{bt04} Let $M$ be a closed subspace of a reproducing
kernel Hilbert space $\mathcal{H}(K).$ Then $M$ is a reproducing kernel Hilbert space with reproducing kernel
\begin{equation*}K_M(z, w)= P_M \big [ K(\cdot, w) \big ] (z),\end{equation*} 
where $P_M$ denotes the orthogonal
projection onto $M$. \end{thm}

We also require results about multipliers of reproducing kernel Hilbert spaces. The following is a
special case of Theorem 2.3.9 in \cite{alp01}.

\begin{thm} \label{thmA.5} Let $\mathcal{H}(K_1)$ and $\mathcal{H}(K_2)$ be reproducing kernel Hilbert spaces on $\Omega,$ and let
$ \phi: \Omega \rightarrow \mathbb{C}.$ Then $M_{\phi}$ is a bounded operator from $\mathcal{H}(K_1)$ to $\mathcal{H}(K_2)$
with norm less than or equal to $b$ if and only if 
\begin{equation*} K_2(z, w) -
\frac{1}{b^2}\phi(z)\overline{\phi(w)} K_1(z, w) \end{equation*} is a positive kernel.\end{thm}

The following is a corollary of Theorem \ref{thmA.5}:

\begin{cor} \label{corA.1} Let $\mathcal{H}(K_1)$ and $\mathcal{H}(K_2)$ be reproducing kernel Hilbert spaces on
$\Omega.$ Then, the space $\mathcal{H}(K_1)$ is contained in $\mathcal{H}(K_2)$ if and only if there is some constant $b$
such that difference 
\begin{equation*}K_2( z, w) -\frac{1}{b^2} K_1(z, w)\end{equation*} is a positive kernel on
$\Omega$. If $ b \le 1$, the containment is contractive. \end{cor}

\end{document}

%% file: def_latex.tex
\usepackage{latexsym}
\usepackage{amssymb}
\usepackage{euscript}
\usepackage[dvips]{graphics}
\usepackage{epsf}

\def\mcc{M\raise.5ex\hbox{c}C}
\def\mccarthy{M\raise.5ex\hbox{c}Carthy}

\def\={\ = \ }

\def\be{\setcounter{equation}{\value{theorem}} \begin{equation}}
\def\ee{\end{equation} \addtocounter{theorem}{1}}
\def\beq{\begin{eqnarray*}}
\def\eeq{\end{eqnarray*}}

\def\bl{\begin{lemma}}
\def\el{\end{lemma}}
\def\bt{\begin{theorem}}
\def\et{\end{theorem}}
\def\bprop{\begin{prop}}
\def\eprop{\end{prop}}
\def\bd{\begin{definition}}
\def\ed{\end{definition}}
\def\br{\begin{remark}}
\def\er{\end{remark}}
\def\bexer{\begin{exercise}}
\def\eexer{\end{exercise}}
\def\bfig{\begin{figure}}
\def\efig{\end{figure}}

\newtheorem{theorem}{Theorem}[section]
\newtheorem{prop}[theorem]{Proposition}
\newtheorem{lemma}[theorem]{Lemma}
\newtheorem{cor}[theorem]{Corollary}

\newtheorem{definition}[theorem]{Definition}

\newtheorem{remark}[theorem]{Remark}